\documentclass[11pt]{article}
\usepackage{mathrsfs}
\usepackage{amssymb,latexsym,tikz,graphicx,eepic}
\usepackage{amsthm}
\usepackage{amsmath,mathrsfs,amssymb}
\usepackage{amscd,mathtools}
\usepackage{rotating}
\usepackage{amsfonts}
\usepackage{graphicx}
\usepackage[all]{xy}
\usepackage[english] {babel}

\usepackage[utf8]{inputenc}

\usepackage{mathtools}

\usetikzlibrary{arrows,shapes}

\usepackage[top=0.5in,bottom=1in,left=1.25in,right=1.25in]{geometry}
\usepackage{graphicx}
\usepackage{verbatim}

\usepackage{setspace}
\usepackage{titlesec}
\usepackage{stmaryrd}

\newtheorem{theorem}{Theorem}[section]
\newtheorem*{theorem*}{Theorem}

\newtheorem{remark}[theorem]{Remark}
\newtheorem{proposition}[theorem]{Proposition}
\newtheorem{corollary}[theorem]{Corollary}

\newtheorem{lemma}[theorem]{Lemma}

\newtheorem{theoremintr}{Theorem}

\numberwithin{equation}{section}

\def\Q{\mathbb{Q}}
\def\Z{\mathbb{Z}}

\def\i{\sqrt{-1}}

\newcommand{\Mod}[1]{\ (\mathrm{mod}\ #1)}

\newcommand*\xbar[1]{%
	\hbox{%
		\vbox{%
			\hrule height 0.5pt 
			\kern0.5ex
			\hbox{%
				\kern-0.1em
				\ensuremath{#1}%
				\kern-0.1em
			}%
		}%
	}%
}
\def\cSn{\xbar{\mathfrak{S}}_n}

\begin{document}

	\title{Special values of spectral zeta functions of graphs and Dirichlet $L$-functions
		\footnote{This work is supported by the National Natural Science Foundation of China (Grant 12071371).}
	}
	\author{Bing Xie\footnote{B. Xie, School of Mathematics and Statistics, Shandong University, Weihai 264209, China.
			Email: {\tt xiebing@sdu.edu.cn}}.,
		Yigeng Zhao\footnote{Y.G. Zhao, School of Science, Westlake University, Hangzhou 310024, China.
			Email: {\tt zhaoyigeng@westlake.edu.cn   }}.,
		and Yongqiang Zhao\footnote{Y.Q. Zhao, School of Science, Westlake University, Hangzhou 310024, China.
			Email: {\tt zhaoyongqiang@westlake.edu.cn  }}.
	}
	\date{\today}

	\maketitle
	
	\begin{abstract}
		In this paper, we establish relations between special values of Dirichlet $L$-functions and those of spectral zeta functions or $L$-functions of cycle graphs. In fact, they determine each other in a natural  way.
		These two kinds of special values are bridged together by a combinatorial derivative formula obtained from studying spectral zeta functions of the first order self-adjoint differential operators on the unit circle.

		\vspace{04pt}
		
		\vspace{04pt}

	\end{abstract}

	\section{Introduction}
	\subsection{Special values of Riemann zeta function and of Dirichlet $L$-functions}
	
	In 1735, Euler proved that the special values of the Riemann zeta function for any  $n\in \mathbb{N}$
	\begin{equation}\label{equ:Eulerzeta}
		\zeta(2n):=\sum_{k=1}^\infty\frac{1}{k^{2n}} =\frac{(-1)^{n+1}(2\pi )^{2n}}{2\cdot (2n)!}B_{2n}\in \pi^{2n}\mathbb{Q},
	\end{equation}
	where $B_k$ denotes the $k$-th Bernoulli number, given via the generating function
	\[ \frac{t}{e^t-1}=\sum_{k=0}^{\infty} \frac{B_kt^k}{k!}.\]
	He also showed for a special Dirichlet $L$-series 
	\begin{equation}\label{equ:EulerDirichlet}
		L(2n+1,\chi_4):=\sum_{k=0}^\infty\frac{(-1)^k}{(2k+1)^{2n+1}} =(-1)^n\frac{\pi^{2n+1}E_{2n}}{2^{2n+2}(2n)!}\in \pi^{2n+1}\mathbb{Q},
	\end{equation}
	where $E_{2n}$ is the $2n$-th secant number, given by the formula
	\[ \sec(t)=   \sum_{n=0}^{\infty} \frac{(-1)^nE_{2n}}{(2n)!}t^{2n}. \]
	In 1940, Hecke \cite{Hecke} first extended the above results to Dirichlet $L$-series
	\begin{equation*}
		L(s,\chi):=\sum_{k=1}^{+\infty}\frac{\chi(k)}{k^s},
	\end{equation*}
	for $s\in\mathbb{C}$ with $\Re(s)>1$ and  real quadratic Dirichlet character $
	\chi: (\mathbb{Z}/N\mathbb{Z})^{\times}\rightarrow \mathbb{C}^\times.
	$
	For a general Dirichlet character, Leopoldt \cite{Leop} studied the values of $L(n,\chi)$, using the generalize Bernoulli numbers, in the case that $n$ and $\chi$ satisfy the parity condition (i.e., $\chi(-1)=(-1)^n$) in 1958.
	More precisely, if $\chi \Mod N$ is a primitive character  and $n$ a positive integer satisfying the parity condition $\chi(-1)=(-1)^n$, then
	\begin{equation}\label{Leopoldt}
		L(n, \chi)=\frac{-\chi(-1)G(\chi)(2\pi \i)^n }{2N n!}\sum^{N}_{j=1}\bar{\chi}(j){\mathbf B}_n(\frac{j}{N}),
	\end{equation}
	where $G(\chi)= \sum_{1\leq a \leq N} \chi(a)e^{\frac{2\pi \i a}{N}}$ is the Gauss sum and ${\mathbf B}_n(x)=\sum_{0\leq j \leq n} {n \choose j} B_j  x^{n-j}$ the $n$-th Bernoulli polynomial.	
	
	Note that formula \eqref{Leopoldt} only holds for primitive characters!  For a non-primitive character,  one needs to lift it to the primitive case to find the special values.  In our 
	Theorems \ref{thm:zeta-values} and \ref{thm:L-to-LN}  below,  we give new formulae for special values of Dirichlet $L$-function which work for arbitrary characters.	
	\vskip 3mm
	\subsection{Spectral zeta functions of cycle graphs $\Z/N\Z$}	
	The \textit{cycle graph} $C_N$ is a simple graph that consists $N$ vertices connected in a closed chain, which might be visualized as the  $N$-polygon in a plane. It is one of the most fundamental graphs in graph theory, and is also one of the simplest and the most basic examples of Cayley graphs. Recall,  a \textit{Cayley graph} $(G, S)$ with a given group $G$ and a symmetric generating set $S\subseteq G$ (i.e. $S=S^{-1}$) is the graph which has the elements of $G$ 
	as vertex sets and two vertices $g_1$ and $g_2$ are adjacent if there is an element $s\in S$ such that $g_1=g_2s$. Note that the cycle graph $C_N$ is just the Cayley graph with $G=\Z/N\Z$ and generating set $S=\{1,-1\}$.	
	\vskip 3mm
	
	\noindent For the cycle graph $C_N$, its Laplacian operator on $L^2(V)$, $V=\Z/N\Z$ the vertex set,  is given by
	\begin{equation}
		\Delta f(x):=\frac{1}{4}(2f(x)-f(x+1)-f(x-1)),
	\end{equation}
	for any $f \in L^2(V)$ and $x\in \Z/N\Z$. For $m\in \{1,2,\cdots, N-1\}$, a direct computation shows that
	\[ \eta_m=(e^{\frac{2m\pi \i}{N}},   e^{\frac{4m\pi \i}{N}}, \cdots,e^{\frac{2Nm\pi \i}{N}}) \]
	is  the eigenvector associated  with the eigenvalue $\sin^2(\frac{m\pi}{N})$. The spectral zeta function of the graph $C_N$  is defined as 
	\begin{equation*}
		\zeta_{\Z/N\Z}(s)=\sum_{m=1}^{N-1} \frac{1}{\sin^{2s}(\frac{m\pi}{N})}.
	\end{equation*}
	For any Dirichlet character $\chi: (\mathbb{Z}/N\mathbb{Z})^{\times}\rightarrow \mathbb{C}^\times$, we define the Dirichlet $L$-function for the cycle graph $C_N$ as
	\begin{align*}
		L_{\Z/N\Z}(s,\chi)=\sum_{m=1}^{N-1} \frac{\chi(m)}{\sin^{2s}(\frac{m\pi}{N})}.
	\end{align*}
	As an analogue of graph $L$-function, we also introduce the following notation
	\begin{align*}
		\widetilde{L}_{\Z/N\Z}(s,\chi)=\sum_{m=1}^{N-1} \frac{\chi(m)\cot (\frac{m\pi}{N})}{\sin^{2s}(\frac{m\pi}{N})}.	
	\end{align*}	
	The spectral zeta functions for cycle graphs $C_N$ and Dirichlet $L$-functions of  cycle graphs are not new and have already appeared in the literature. For example,  they are main subjects of studies in the recent 
	interesting papers \cite{KN06}, \cite{CJK},  \cite{Fri-Karl} and \cite{Friedli}. The main points of \cite{Fri-Karl} and \cite{Friedli} are to use $\zeta_{\Z/N\Z}(s)$ and $L_{\Z/N\Z}(s,\chi)$ to approximate the classical Riemann zeta function $\zeta(s)$ 
	and Dirichlet $L$-function $L(s, \chi)$, respectively. Especially, they reformulate the (generalized) Riemann Hypothesis in terms of these graph zeta functions or $L$-functions.
	Friedli and Karlsson emphasized that the spectral zeta function $\zeta_{\Z/N\Z}(s)$ is a more natural approximation of $\zeta(s)$ than other artificial ones,  see 
	\cite[Pages 590-591]{Fri-Karl}. Here we give another support for the importance of studying these functions from the point of view of special values. 
	Indeed, we find the special values of \eqref{Leopoldt} can also be expressed as those of the spectral zeta function of the graph $\Z/N\Z$.   	
	One of our main results is the following  formulae (cf. Theorem \ref{thm:L-to-LN} and Proposition \ref{spectral-to-zeta-value}):
	
		\begin{theoremintr}\label{main-thm-Dirichlet}
			For any even character $\chi   \Mod N$ and any integer $n\geq 1$, we have
			\begin{align*}
				L(2n,\chi)=\frac{\pi^{2n}}{ 2(2n-1)!N^{2n}} \sum_{i=1}^{n}a_{n,i} L_{\Z/N\Z}(i,\chi),
			\end{align*}
		and
		\begin{align*}
			 L_{\Z/N\Z}(n,\chi)=2\sum_{i=1}^{n}c_{n,i}\frac{N^{2i}}{\pi^{2i}}L(2i,\chi),
		\end{align*}
		where each $c_{n,i}$ is the coefficient of $x^{-2i}$ in the Laurent expansion of  $(\sin x)^{-2n}$ and  $a_{n, i}$ satisfy
			\[  \sum_{i=0}^{n-s}a_{n,n-i}c_{n-i,s}=\begin{cases*}
			(2n-1)! & if $s=n$,\\
			0 & if $1\leq s \leq n-1$,
		\end{cases*}\]
		and can be explicitly given in a combinatorial way ( see \eqref{eq:coefficient_a_ni} in Section \ref{spectal-zeta}).
	\end{theoremintr}

	
Special values $L(2n+1,\chi)$ for an odd character are related to the  special values $\widetilde{L}_{\Z/N\Z}(s,\chi)$ similarly.

As a consequence, we also have similar relations between special values of the Riemann zeta function  and of the spectral zeta functions of the cycle graphs (cf. Theorem \ref{thm:zeta-spectral values}) : 
	
	\begin{theoremintr} 
		For any integers $N\geq 2$ and $n\geq 1$, we have
		\begin{align*}
			(N^{2n}-1)\zeta(2n)=\frac{\pi^{2n}}{2(2n-1)!}\sum_{i=1}^{n}a_{n,i} \zeta_{\Z/N\Z}(i),
		\end{align*}
	and
	\begin{equation*}
	\zeta_{\Z/N\Z}(n)=2\sum_{i=1}^{n}c_{n,i}\frac{(N^{2i}-1)}{\pi^{2i}}\zeta(2i).
\end{equation*}	
		where $a_{n, i}$ and $c_{n,i}$  are the same as  given in Theorem \ref{main-thm-Dirichlet}.
	
	\end{theoremintr}

	\subsection{First order self-adjoint differential operators}
	
	Our project begins with a study of the spectral zeta function of the first order self-adjoint
	differential operators on the unit circle $\mathbb{S}^1$ of the form
	\begin{equation}\label{one-order-G}
		T_v u:=-\i u'+v u=\lambda wu 
	\end{equation}
	with the weighted function $w \in L^{1} [0,1]$, $w> 0$, a.e. on $[0,1]$,
	and the potential function $v \in L^{1} [0,1] $. Its $k$-th eigenvalue is
	\begin{equation}\label{Lambda-k}
		\lambda_k=\frac{2k\pi+\int_0^1vdx}{\int_0^1 wdx}=:\frac{1}{c}(2k\pi+\alpha),\;k=0,\pm 1,\pm2, \cdots.
	\end{equation}
	For $\alpha\notin 2\pi\mathbb{Z}$, we can define the spectral function
	\begin{equation}\label{Spectral-Fun}
		\zeta_{T_v} (s):=\sum_{k=-\infty}^{+\infty}\frac{1}{\lambda_k^s},\; \Re(s)> 1.
	\end{equation}
	For any $s=n\in\mathbb{N}\backslash\{0\}$, thanks to Mercer's Theorem (cf. \cite[\S 3.5.4]{Courant-Hilbert}),
	we can relate the spectral series in \eqref{Spectral-Fun} with the integral of Green function of the differential operator.
	For $n\geq 3$, the calculation of the integral naturally involves classical combinatorial problems on counting  $n$-permutations.
	
	Using these facts, we study the special values of the series
	\begin{equation}\label{Series-1}
		\sum_{k=-\infty}^{+\infty}\frac{1}{(2k\pi+\alpha)^n},
		\;n\geq 1,\;\frac{\alpha}{2\pi}\notin \mathbb{Z}.
	\end{equation}
	By giving a new combinatorial interpretation, we deduce an explicit formula for this sum.
	We systematically use the classical Mercer's theorem for a differential operator, which can be viewed as a baby version of the trace formula,
	to evaluate special values.  As a corollary, we prove the above known results in a new simple way. 	
	In the particular case $\alpha=\pi/2$, we recover the above result of Euler. 
	Here, we  only consider the first order differential operator case and will treat the second order operators, i.e., Sturm-Liouville operators in \cite{xzz}.
	
	Series \eqref{Series-1} are closely related to the special values of Hurwitz zeta function and Dirichlet  $L$-functions (cf. \cite{Dav}),
	which arise out of number theory problems and other considerations.  Using the explicit formula for series \eqref{Series-1}, we deduce special
	values of Riemann zeta function and Dirichlet $L$-series, which recover Euler and Hecke's classical results.
	\vskip 3mm

	\section{Special values of spectral zeta functions of the first order differential operators}
	
	In this section we study special values of spectral zeta functions of the first order self-adjoint 
	differential operators on the unit circle $\mathbb{S}^1$  of the form \eqref{one-order-G}. It turns out
	its spectrum is quite simple and one can easily write down its spectral zeta function. Then we use the 
	classical Mercer's theorem to evaluate integral values of the spectral zeta function. This naturally 
	leads to the combinatorial problems on counting permutations with fixed difference number between descents and ascents. 
	
	In subsection \ref{Pre}, we recall basic facts of first order differential operators, especially, Green functions and Mercer's theorem.
	Basic combinatorial facts on counting permutations will be recalled in the next subsection. In the final subsection \ref{derivative}, 
	we give the main theorem of this section. 
	
	\subsection{Green function and Mercer's Theorem} \label{Pre}
	
	The first order self-adjoint differential operator	\eqref{one-order-G} is equivalent to the
	boundary value problem
	\begin{equation}\label{1-order-G}
		T_v u=-\i u'+v u=\lambda w u,\;{\rm on}\;(0,1),
	\end{equation}
	with the boundary condition
	\begin{equation}\label{one-order-BC}
		u(0)=u(1).
	\end{equation}
	Then $T_v$ is a self-adjoint operator with the domain
	$$
	\mathcal{D}(T_v)=\{u\in AC[0,1]:\;\frac{1}{w}T_v u\in L_w^2[0,1]\;{\rm and}\;u(0)=u(1)\},
	$$
	where $AC[0,1]$ represents all the absolutely continuous functions on $[0,1]$
	and
	$$
	L_w^2[0,1]:=\{u\;{\rm is\;measurable}:\;\int_0^1 |u|^2w dx<\infty\}.
	$$
	The $k$-th eigenvalue of problem \eqref{1-order-G} and \eqref{one-order-BC} is the same as
	$\lambda_k$ in \eqref{Lambda-k},
	$$ \lambda_k=\frac{1}{c}(2k\pi+\alpha),  \;\;  k=0,\pm 1,\pm2, \cdots,$$
	and the corresponding eigenfunction is
	$$\varphi_k(x)=e^{\i\int_0^x(\lambda_k w(t)-v(t)) dt}.$$
	
	Note $c\lambda_k$ is a translation of $2k\pi$ by $\alpha=\int_0^1v dx$, which is related
	to the integral of the potential function $\int_0^1v dx$. Thus, for $c=\int_0^1 w dx>0$ the operator
	$cT_v$ has the same eigenvalues with
	\begin{equation}\label{one-order}
		Tu=-\i u'+\alpha u=\lambda u,\;{\rm on}\;(0,1),
	\end{equation}
	with the boundary condition \eqref{one-order-BC}. We have obtained the following proposition.
	\begin{proposition}\label{Prop-Potential}
		For any given weighted function $w$ with $c=\int_0^1 w dx$ and  potential function
		$v(x)$ with $\alpha=\int_0^1 v dx$, the corresponding self-adjoint
		operators $cT_v$ and $T$ have the same eigenvalues $\lambda_k=2k\pi+\alpha$ and same spectral
		zeta function $\zeta_T (s)=\sum_{k=-\infty}^{+\infty}\frac{1}{\lambda_k^s}$, $\Re(s)\geq 1$,
		where the function value $\zeta_T (1)$ is understood as the limit
		$$
		\sum_{k=-\infty}^{+\infty}\frac{1}{\lambda_k}:=\lim_{N\rightarrow +\infty}\sum_{k=-N}^{N}\frac{1}{\lambda_k}
		=\lim_{N\rightarrow +\infty}\sum_{k=-N}^{N}\frac{1}{2k\pi+\alpha}.
		$$
	\end{proposition}
	Therefore, without loss of generality, we consider the boundary value problems \eqref{one-order} with
	the boundary condition \eqref{one-order-BC}.
	In the following, we always assume that $\alpha\neq 2k\pi$, $k=0,\pm 1,\pm2, \cdots$,
	then $0$ is not the eigenvalues of $T$. Hence $T^{-1}$ exists and is a bounded linear operator on $L^2[0,1]$,
	and the Green function $G(s,t)$ of \eqref{one-order} at $0$ is defined as that for any $f\in L^2[0,1]$,
	\begin{equation}\label{T-Green}
	T^{-1}f(s)=\int_0^1G(s,t)f(t){\rm d}t,\;s\in[0,1]. 
	\end{equation}
	The definition is equivalent to, for any fixed $ t\in[0,1]$, 
	$$
	TG(s,t)=\delta_t(s),
	$$
	where $\delta_t(s)$ is the Delta function at $t$.  Hence the Green function satisfies
	$$
	-\i[G(t+,t)-G(t-,t)]=1,
	$$
	for any $t\in(0,1)$,  where $G(t\pm,t):=\lim_{\varepsilon\rightarrow 0+}G(t\pm\varepsilon,t)$.
	By the definition, we get the Green function of $T$ at $0$,
\begin{align}\label{Green-1-T1}
	G(s,t)=
		\begin{cases} 
			\frac{e^{\i\alpha(t-s)}}{2\sin(\alpha/2)	}e^{-\frac{1}{2}\i\alpha}, & 0\leq s< t\leq 1,\\
				\frac{e^{\i\alpha(t-s)}}{2\sin(\alpha/2)	}e^{\frac{1}{2}\i\alpha}, & 0\leq t< s\leq 1.
			\end{cases}
	\end{align}

	Using eigenfunctions $\{\varphi_k(x)=e^{2k\pi \i x}$, $k=0,\pm 1,\pm2, \cdots\}$ of the operator $T$, we  get
	a series representation of $G(s,t)$. In fact, for any fixed $s\in[0,1]$, we can expand $G(s,t)$ by
	Fourier series
	$$
	G(s,t)
	=\sum_{k=-\infty}^{+\infty}\phi_k(s)\overline{\varphi_k(t)},
	$$
	where the Fourier coefficient
	$$
	\phi_k(s)=\langle G(s,t),\overline{ \varphi_k(t)}\rangle:=\int_0^1 G(s,t)\varphi_k(t){\rm d}t
	=\frac{\varphi_k(s)}{\lambda_k},
	$$
	by \eqref{T-Green}. We have obtained
	\begin{equation}\label{Green-Series}
		G(s,t)
		=\sum_{k=-\infty}^{+\infty}\frac{\varphi_k(s)\overline{\varphi_k(t)}}{\lambda_k}
		=\sum_{k=-\infty}^{+\infty}\frac{\varphi_k(s)\overline{\varphi_k(t)}}{2k\pi+\alpha}
		=\sum_{k=-\infty}^{+\infty}\frac{\eta^k}{2k\pi+\alpha},
	\end{equation}
	where $\eta:=\eta(s,t):=e^{2\i\pi(t-s)}$.
	
	Note for any fixed $s$ (or $t$), due to the discontinuity of the Green function $G(s,t)$,
	the series \eqref{Green-Series} is  not uniformly convergent. We can use the method in \cite[\S 2.1]{Weil}
	to deal with this issue. Let‘s consider the limit
	$$
	G(s,t)
	=\lim_{N\rightarrow +\infty}\sum_{k=-N}^{N}\frac{\varphi_k(s)\overline{\varphi_k(t)}}{\lambda_k}.
	$$
	For any given  $\varepsilon>0$, we have
	\begin{align*}
		\frac{1}{2}\left[G(t+\varepsilon,t)+G(t-\varepsilon,t)\right]
		&=\Re\left[ G(t+\varepsilon,t)\right]\\
		&=\Re\left[
		\lim_{N\rightarrow +\infty}\sum_{k=-N}^{N}\frac{\varphi_k(t+\varepsilon)\overline{\varphi_k(t)}}{\lambda_k}
		\right]\\
		&=\frac{1}{\alpha}+\sum_{k=1}^{\infty}\frac{-2\alpha\cos(2k\pi\varepsilon)}{(2k\pi)^2-\alpha^2},
	\end{align*}
	and the last series uniformly converges. Therefore,
	$$
	\frac{1}{2}[G(t+,t)+G(t-,t)]
	=\frac{1}{\alpha}+\sum_{k=1}^{\infty}\frac{-2\alpha}{(2k\pi)^2-\alpha^2}
	=\lim_{N\rightarrow +\infty}\sum_{k=-N}^{N}\frac{1}{2k\pi+\alpha}.
	$$
	Thanks to the uniform convergence,  we apply Mercer's Theorem to the first order differential operator \eqref{one-order} to get
	\begin{equation}\label{T-Mercer}
		\lim_{N\rightarrow +\infty}\sum_{k=-N}^{N}\frac{1}{2k\pi+\alpha}
		=\int_0^1\frac{1}{2}[G(t+,t)+G(t-,t)]{\rm d}t
		=\frac{1}{2}\cot\left(\frac{\alpha}{2}\right).
	\end{equation}

	For $n\geq 2$, Mercer's Theorem (cf. \cite[\S 3.5.4]{Courant-Hilbert}) tells us that
	\begin{align}\label{equ:mercer}
		\sum_{k=-\infty}^{+\infty}\frac{1}{(2k\pi+\alpha)^n}
		&=\int_0^1\cdots\int_0^1G(x_1,x_2)\cdots G(x_n,x_1){\rm d}x_1\cdots{\rm d}x_n.
	\end{align}
	Thus in order to evaluate series \eqref{Series-1}, it is enough to calculate the integral on the right hand side of \eqref{equ:mercer}. 	
	For $n=2$,  we just notice that
	$$
	2^2\sin^2(\frac{\alpha}{2})G(x_1,
	x_2)G(x_2,x_1)=1.
	$$
	Hence,
	\begin{align}\label{T2-Mercer}
		\sum_{k=-\infty}^{+\infty}\frac{1}{(2k\pi+\alpha)^2}
		=\int_0^1\int_0^1G(x_1,x_2)G(x_2,x_1){\rm d}x_1{\rm d}x_2
		=\frac{1}{2^2\sin^2(\alpha/2)}.
	\end{align}	
	
	However, to calculate the right hand side of (\ref{equ:mercer}) for general $n\geq 3$, we need to compare the consecutive $x_j$ and $x_{j+1}$ and then get an explicit integrand. More precisely, by noting that
	$$
	e^{\i\alpha(x_{1}-x_{2})}\cdots e^{\i\alpha(x_{n}-x_{1})}=1,
	$$
	we have
	\begin{equation}\label{equ:Green}
		G(x_1,x_2)\cdots G(x_n,x_1)=\prod_{j=1}^{n}G(x_j, x_{j+1})=\frac{1}{2^n\sin^{n}(\frac{\alpha}{2})}e^{\frac{1}{2}\i\alpha (\sharp\{j | x_j>x_{j+1}\}-\sharp\{ j | x_j<x_{j+1}\})}
	\end{equation}
	where we set $x_{n+1}=x_1$ for our convenience, and $\sharp A$ means the cardinality of a finite set $A$.
	This comparison problem naturally leads to a classical combinatorial problem on counting permutations.
	
	\subsection{Combinatorics}
	
	A linear ordering of the elements
	of the set $\{1,\cdots,n\}$ is called an $n$-permutation and denote the set of all
	$n$-permutations  $\sigma= \sigma(1)\cdots\sigma(n) $ as $\mathfrak{S}_n$. If we set $\sigma(0)=\sigma(n)$ and $\sigma(n+1)=\sigma(1)$,
	then it is called a \textit{circular $n$-permutation}, and denoted as $\sigma=(\sigma(1),\cdots,\sigma(n) )$. The set of all circular $n$-permutations
	is denoted by  $\cSn$.
	For  $\sigma=(\sigma(1), \cdots,\sigma(n)) \in \cSn$,  we call $\sigma(i)$ is an \textit{ascent}
	(resp. \textit{descent}) if  $ \sigma(i)<\sigma(i+1)$ (resp.  $\sigma(i)>\sigma(i+1)$).
	
	To determine the integrand in \eqref{equ:mercer}, we need to compute the following  class number $M_m$ for  all possible $2-n\leq m \leq n-2$:
	
	$$
	M_m:=\sharp\left\{\sigma\in\cSn \bigg|\;m
	=\sharp\{ {\rm descents\; of}\; \sigma\}-\sharp\{ {\rm ascents\; of}\; \sigma\}\right\}.
	$$
	In the following, we show this class number is just the classical Eulerian numbers.
	Let
	$$
	A(n,l):=\sharp\left\{\sigma\in\mathfrak{S}_n \bigg|\;\sharp\{ {\rm descents\; of}\; \sigma\}=l-1\right\},
	$$
	and
	$$
	\overline{A}(n,l):=\sharp\left\{\sigma\in\cSn \bigg|\;\sharp\{ {\rm descents\; of}\; \sigma\}=l\right\},
	$$
	where the definition of $A(n,l)$ is given in \cite[\S1.1.2, p.4]{Bona}.
	
	\begin{proposition}( {\cite[Theorem 1.11]{Bona}})\label{Bona-Thm}
		For integers $n\geq 1$ and $l$ satisfying $l\leq n$,
		we have
		$$
		A(n,l)=\sum_{i=0}^{l} (-1)^i{n+1 \choose i}
		(l-i)^{n}.
		$$
	\end{proposition}

	For any $\sigma\in\cSn $ with $\sharp\{ {\rm descents\; of}\; \sigma\}=l$, $\sigma$ can be
	written as $\sigma=(n,\sigma_{n-1})$, where $\sigma_{n-1}\in\mathfrak{S}_{n-1}$
	with $\sharp\{ {\rm descents\; of}\; \sigma_{n-1}\}=l-1$. Hence $\overline{A}(n,l)=A(n-1,l)$
	and using this proposition, we can get the similar conclusion about
	$\overline{A}(n,l)$.
	
	\begin{lemma}\label{A-Property}
		For  integers $n\geq 2$ and $l$ satisfying $l\leq n-1$, we have
		\begin{itemize}
			\item[(i)] $\overline{A}(n,l)=\sum\limits_{i=0}^{l} (-1)^i{n \choose i}
			(l-i)^{n-1}.
			$
			\item[(ii)] $\overline{A}(n,l)=\overline{A}(n,n-l). $
		\end{itemize}	
	\end{lemma}
	\begin{proof}
		We only need to prove (ii). Note for any $\sigma\in\cSn$,
		$$
		\sharp\{ {\rm ascents\; of}\; \sigma\}+\sharp\{ {\rm descents\; of}\; \sigma\}=n
		$$
		and
		$$
		\sharp\{ {\rm ascents\; of}\; \sigma\}=\sharp\{ {\rm descents\; of}\; R(\sigma)\},
		$$
		where $R(\sigma)=(\sigma(n),\sigma(n-1),\cdots,\sigma(1))$ is the reversal of $\sigma\in\cSn$. Hence
		$$
		\sharp\{ {\rm descents\; of}\; R(\sigma)\}=n-\sharp\{ {\rm descents\; of}\; \sigma\}.
		$$
		Therefore, the claim follows.
	\end{proof}

	For any $\sigma\in\cSn $ with $\sharp\{ {\rm descents\; of}\; \sigma\}=l$, we have
	$$
	m(\sigma):=\sharp\{ {\rm descents\; of}\; \sigma\}-\sharp\{ {\rm ascents\; of}\; \sigma\}=2l-n ,
	$$
	which is equivalent to $ l=(n+m)/2$.
	Therefore, we have
	$$ M_m = \overline{A}(n,(n-m)/2). $$
	\vskip 3mm
	
	\subsection{A combinatorial derivative formula}\label{derivative}
	Given $(x_1,\cdots, x_n)\in [0,1]^n$, we may assume $x_i\neq x_j$, for any $i\neq j$. Then there exists a unique $\sigma\in \mathfrak{S}_n$ such that
	\[ x_{\sigma(1)} <x_{\sigma(2)}<\cdots < x_{\sigma(n)}.\]
	Therefore, by \eqref{equ:Green}, we have
	\[ \prod_{i=1}^nG(x_i,x_{i+1})=\frac{1}{2^n\sin^{n}(\frac{\alpha}{2})}e^{-\frac{1}{2}\i\alpha\cdot m(\sigma)}.\]
	It follows that
	\begin{align*}
		&\int_0^1\cdots\int_0^1G(x_1,x_2)\cdots G(x_n,x_1){\rm d}x_1\cdots{\rm d}x_n\\
		&=\frac{1}{2^n\sin^{n}(\frac{\alpha}{2})}\sum_{\sigma \in \mathfrak{S}_n} \int_0^1\int_{0}^{x_{\sigma(n)}}\cdots\int_{0}^{x_{\sigma(2)}}{e^{-\frac{1}{2}\i\alpha\cdot m(\sigma)}}dx_{\sigma(n)}dx_{\sigma(n-1)}\cdots dx_{\sigma(1)}\\
		&=\frac{n}{2^n\sin^{n}(\frac{\alpha}{2})}\sum_{\sigma \in \cSn} \int_0^1\int_{0}^{x_{\sigma(n)}}\cdots\int_{0}^{x_{\sigma(2)}}{e^{-\frac{1}{2}\i\alpha\cdot m(\sigma)}}dx_{\sigma(n)}dx_{\sigma(n-1)}\cdots dx_{\sigma(1)}\\
		&=\frac{n}{2^n\sin^{n}(\frac{\alpha}{2})}\sum_{m=2-n}^{n-2} M_m e^{-\frac{1}{2}\i\alpha m}\int_{0}^1 \int_{0}^{x_n}\cdots \int_{0}^{x_2}dx_ndx_{n-1}\cdots dx_1\\
		&=\frac{1}{2^n\sin^n(\frac{\alpha}{2})}\frac{1}{(n-1)!}
		\sum_{l=1}^{n-1}e^{-\i(n/2-l)\alpha}\overline{A}(n,l).
	\end{align*}
	Using Lemma \ref{A-Property} (ii), we may further simplify the terms in the above sum:
	\begin{equation}\label{Temp-torus}
		e^{-\i(n/2-l)\alpha}\overline{A}(n,l)+e^{-\i(n/2-(n-l))\alpha}\overline{A}(n,n-l)
		=2\cos((n/2-l)\alpha)\overline{A}(n,l).
	\end{equation}
	In summary, our main result for this section is the following theorem.
	\begin{theorem}\label{Mercer-torus}
		For any $\alpha\neq 2k\pi$, $k=0,\pm 1,\pm2, \cdots$,
		and $n\geq 2$, we have
		\begin{align}\label{Mercer-torus-temp}
			\sum_{k=-\infty}^{+\infty}\frac{1}{(2k\pi+\alpha)^n}
			=\frac{1}{2^n (n-1)!\sin^n(\alpha/2)} \sum_{l=1}^{n-1}\cos\left((n/2-l)\alpha\right)\overline{A}(n,l),
		\end{align}
		and for $n=1$, we have the well known formula
		$$
		\lim_{N\rightarrow +\infty}\sum_{k=-N}^{N}\frac{1}{2k\pi+\alpha}
		=\frac{1}{2}\cot\left(\frac{\alpha}{2}\right).
		$$
	\end{theorem}
	
	\begin{remark}
		Theorem \ref{Mercer-torus} is indeed a (combinatorial) higher order derivative formula for $\cot x$, i.e.,
		$$\frac{(-1)^{n-1}}{(n-1)!}\frac{d^{n-1}}{d\alpha^{n-1}}\left(\frac{1}{2}\cot\frac{\alpha}{2}\right)= \sum_{k=-\infty}^{+\infty}\frac{1}{(2k\pi+\alpha)^n}.$$       
		Once one has known the expression in \eqref{Mercer-torus-temp}, it would be possible
		to prove Theorem \ref{Mercer-torus} directly by induction,  and by using the induction formula of $\overline{A}(n,l)(cf. $ \cite[Theorem 1.7]{Bona}) and playing with triangle identities. We, however, prefer to 
		keep this seemingly more complicated way to establish the result for two reasons: first, this is the way how we found this expression exactly; 
		second,  the idea of using Green function, Mercer's theorem and combinatorics might be useful for higher-order operators or higher-dimension situations.
	\end{remark}

	\noindent In particular, let $\alpha=\pi/2$ in Theorem \ref{Mercer-torus},
	and we get the following identities: for $n\geq 2$, 
	\begin{equation}\label{equ:4n+1}
		S(n):=\sum_{k=-\infty}^{+\infty}\frac{1}{(4k+1)^n}
		=\frac{\pi^n}{2^n2^{n/2}}\frac{1}{(n-1)!}
		\sum_{l=1}^{n-1}\cos\left((n/2-l)\frac{\pi}{2}\right)\overline{A}(n,l);
	\end{equation}
	and for $n=1$,
	$$
	\sum_{k=0}^\infty\frac{(-1)^k}{2k+1}
	=\lim_{N\rightarrow +\infty}\sum_{k=-N}^{N}\frac{1}{4k+1}=\frac{\pi}{2}\cdot \frac{1}{2}\cot\left(\frac{\pi}{4}\right)
    =\frac{\pi}{4}
	$$
   
	by \eqref{T-Mercer}. 	Note that 
	\[ S(n)=\begin{cases}
		(1-2^{-n})\zeta(n) \;\; &\text{if $n$ is even},\\
		L(n,\chi_4) \;\; &\text{if  $n$  is odd}.
	\end{cases}
	\]
	Comparing Euler's explicit formulas for $\zeta(2n)$ and $L(2n+1,\chi_4)$  and Theorem \ref{Mercer-torus},  we have the following identities.
	\begin{corollary}
		The $2n$-th Bernoulli number
		\begin{equation}
			B_{2n}=\frac{(-1)^{n+1}n}{2^{3n-2} (2^{2n}-1)}\sum_{l=1}^{2n-1} \cos(\frac{n-l}{2}\pi)\overline{A}(2n, l),
		\end{equation}
		and the $2n$-th secant number
		\begin{equation}
			E_{2n}=(-1)^n\frac{\sqrt{2}}{2^{n}}\sum_{l=1}^{2n} \cos\left(\frac{n-l+\frac{1}{2}}{2}\pi\right)\overline{A}(2n+1,l).
		\end{equation}
	\end{corollary}

	\section{Special values of Dirichlet $L$-functions}\label{L-fun}
	
	In this section, we study the relationship between the spectral series
	\eqref{Series-1} and  Dirichlet $L$-functions via Hurwitz zeta functions. Recall for $s\in\mathbb{C}$
	with $\Re(s)>1$ the Hurwitz zeta function (cf. \cite[\S 12]{Apostol}) is
	$$
	\zeta(s,a):=\sum_{k=0}^{+\infty}\frac{1}{(k+a)^s},
	$$
	for $a\not= 0,-1,-2,\cdots$.
	For a Dirichlet character \cite[\S 4]{Dav},
	$$
	\chi: (\mathbb{Z}/N\mathbb{Z})^{\times}\rightarrow \mathbb{C}^\times,
	$$
	and for $s\in\mathbb{C}$ with $\Re(s)>1$, the Dirichlet $L$-function is
	\begin{equation}\label{L-Function}
		L(s,\chi):=\sum_{k=1}^{\infty}\frac{\chi(k)}{k^s}.
	\end{equation}
	By rearranging the terms according to the residue classes modulo $N$, we see that $L(s,\chi)$  is a linear combination of Hurwitz zeta functions:
	\begin{equation*}
		L(s,\chi)=\sum_{m=1}^{N}\sum_{k=0}^{\infty}\frac{\chi(kN+m)}{(kN+m)^s}=\frac{1}{N^s}\sum_{m=1}^{N}\chi(m)\sum_{k=0}^{\infty}\frac{1}{(k+\frac{m}{N})^s}=\frac{1}{N^s}\sum_{m=1}^{N}\chi(m)\zeta(s,\frac{m}{N}).
	\end{equation*}
	For our purpose, the above relation can also be written as
	\begin{align}\label{Dir-Hur}
		L(s,\chi)&=\frac{1}{2N^s}\left[\sum_{m=1}^{N}\chi(m)\zeta(s,\frac{m}{N})+\sum_{m=1}^{N}\chi(N-m)\zeta(s,\frac{N-m}{N})\right] \nonumber\\
		&=\frac{1}{2N^s}\left[\sum_{m=1}^{N}\chi(m)\zeta(s,\frac{m}{N})+\chi(-m)\zeta(s,1-\frac{m}{N})\right]  \nonumber\\
		&=\frac{1}{2N^s}\sum_{m=1}^{N}\chi(m)\left[\zeta(s,\frac{m}{N})+\chi(-1)\zeta(s,1-\frac{m}{N})\right].
	\end{align}

	\noindent Similarly, the spectral zeta function is also a linear combination of Hurwitz zeta functions:
	\begin{align}\label{Spe-Hur}
		\sum_{k=-\infty}^{+\infty}\frac{1}{(2k\pi+\alpha)^n}
		&=\frac{1}{(2\pi)^n}\sum_{k=-\infty}^{+\infty} \frac{1}{(k+\frac{\alpha}{2\pi})^n}\nonumber\\
		&=\frac{1}{(2\pi)^n}\left[\sum_{k=0}^{\infty}\frac{1}{(k+\frac{\alpha}{2\pi})^n}+\sum_{k=1}^{\infty}\frac{1}{(-k+\frac{\alpha}{2\pi})^n}\right] \nonumber\\
		&=\frac{1}{(2\pi)^n}\left[\sum_{k=0}^{+\infty}\frac{1}{(k+\frac{\alpha}{2\pi})^n}+(-1)^n\sum_{k=0}^{+\infty}\frac{1}{(k+1-\frac{\alpha}{2\pi})^n}
		\right]\nonumber\\
		&=\frac{1}{(2\pi)^n}\left[\zeta(n,\frac{\alpha}{2\pi})+(-1)^n\zeta(n,1-\frac{\alpha}{2\pi})\right].
	\end{align}
	
	\noindent For any Dirichlet character $	\chi: (\mathbb{Z}/N\mathbb{Z})^{\times}\rightarrow \mathbb{C}^\times$ and any integer $n\geq 1$ satisfying $\chi(-1)=(-1)^n$. 
	Combining \eqref{Dir-Hur} and \eqref{Spe-Hur}, we have the following identity.		
	
	\[
	L(n,\chi)=\frac{2^{n-1}\pi^n}{N^n}\sum_{m=1}^N\chi(m) \sum_{k=-\infty}^{+\infty}\frac{1}{(2k\pi+\frac{2m\pi}{N})^n}.
	\]
	
	\noindent Thus, as a consequence of our derivative formula Theorem \ref{Mercer-torus}, we get
	\begin{theorem}\label{thm:zeta-values}
		For any Dirichlet character $\chi: (\mathbb{Z}/N\mathbb{Z})^{\times}\rightarrow \mathbb{C}^\times$, $N>2$		
		and any integer $n\geq 2$ satisfying $\chi(-1)=(-1)^n$, we have
		\begin{equation}
			L(n,\chi)=\frac{\pi^n}{2(n-1)!N^n}\sum_{m=1}^{N}\frac{\chi(m)}{\sin^n(m\pi/N)}\sum_{l=1}^{n-1}\cos\big(\frac{(n-2l)m\pi}{N}\big)\bar{A}(n,l).
		\end{equation}
		For $n=1$ and an odd character $\chi$, we have
		\begin{equation}\label{chi1}
			L(1,\chi)=\frac{\pi}{2N}\sum_{m=1}^{N}\chi(m) \cot(\frac{m\pi}{N}).
		\end{equation}
		
	\end{theorem}
	
	\begin{remark}
		The expression for $L(1, \chi)$ \eqref{chi1} is well known, see, for example, {\cite[Proposition 1]{Lou}}. \\
		It is worth mentioning  that  the Dirichlet character in Theorem \ref{thm:zeta-values} is arbitrary instead of 
		primitive as was usually required  in the traditional formulae for $L(n, \chi)$, which makes it easier to use, especially, in problems related to family of special values $L(n, \chi)$.  See, for example, 
		Corollary \ref{cor:mean} in next section.
	\end{remark}
	\vskip 3mm

	\section{Special values of spectral zeta functions for cycle graphs $\Z/N\Z$}\label{spectal-zeta}
	
	In this section, we give Theorem \ref{thm:zeta-values} another interpretation which relates special values of Dirichlet $L$-functions to 
	those of spectral zeta functions of the cycle graphs  $\Z/N\Z$.  Recall the spectral zeta function associated to the Cayley graph $\Z/N\Z$ is
	\begin{equation}
		\zeta_{\Z/N\Z}(s)=\sum_{m=1}^{N-1} \frac{1}{\sin^{2s}(\frac{m\pi}{N})}.
	\end{equation}
	
	The special values of these spectral zeta functions are well studied in mathematical physics and algebraic geometry,  which are called the Verlinde numbers (cf. \cite{Zagier})
	\[ V_g(N)\coloneqq\sum_{m=1}^{N-1} \frac{1}{\sin^{2g}(\frac{m\pi}{N})} \in \mathbb{Q}.\]
	In fact, $\Big(\frac{N}{2}\Big)^g V_g(N) $ is a polynomial in $N$ of degree $3g$.  More precisely, we have the following explicit computation.
	\begin{lemma}(\cite[Theorem 1(iii)]{Zagier}) \label{Lemma_Zagier} We have
		\begin{align}
			\zeta_{\Z/N\Z}(g)=V_g(N)=\sum_{s=0}^{g}\frac{(-1)^{s-1}2^{2s}B_{2s}}{(2s)!}c_{g,s}N^{2s},
		\end{align}
		where $c_{g,s}$ is the coefficients of $x^{-2s}$ in the Laurent expansion of $(\sin x)^{-2g}$ at $x=0$, i.e. 
		$$\frac{1}{\sin^{2g} x}=\frac{c_{g, g}}{x^{2g}}+\cdots + \frac{c_{g, 1}}{x^{2}}+c_{g, 0}+ \cdots .$$
		Note that $c_{g,g}=1$ for any integer $g\in \mathbb{N}$.
	\end{lemma}
	\vskip 3mm
	
	\noindent For any Dirichlet character $\chi: (\mathbb{Z}/N\mathbb{Z})^{\times}\rightarrow \mathbb{C}^\times$, we let
	\begin{align}\label{graph-Dirichlet-L}
		L_{\Z/N\Z}(s,\chi)=\sum_{m=1}^{N-1} \frac{\chi(m)}{\sin^{2s}(\frac{m\pi}{N})}.
	\end{align}
	Note that \eqref{graph-Dirichlet-L}  vanishes identically for an odd character.  To treat the Dirichlet $L$-values at odd integers, we also introduce the following notation
	\begin{align*}
		\widetilde{L}_{\Z/N\Z}(s,\chi)=\sum_{m=1}^{N-1} \frac{\chi(m)\cot (\frac{m\pi}{N})}{\sin^{2s}(\frac{m\pi}{N})}.	
	\end{align*}
	\vskip 3mm
	
	\noindent Recall that \[ \cos (2k\beta)=(-1)^kT_{2k}(\sin \beta); \; \cos((2k+1)\beta)=(-1)^k\cos\beta U_{2k}(\sin\beta), \]
	for any $k\in \mathbb{N}$, where 
	\[T_{2k}(x)=\sum_{r=0}^k{2k \choose 2r}(x^2-1)^rx^{2k-2r}=\sum_{i=0}^{k}(-1)^{k-i}\frac{k}{k+i}{k+i\choose k-i}{(2x)}^{2i}\] 
	is the $2k$-th Chebyshev polynomial of the first kind and
	\[ U_{2k}(x)=\sum_{r=0}^k{2k+1 \choose 2r+1}(x^2-1)^rx^{2k-2r}=\sum_{i=0}^{k}(-1)^{k-i}{k+i\choose k-i}{(2x)}^{2i}\]
	is the $2k$-th Chebyshev polynomial of the second kind.
	For convenience, we write 
	\[ T_{2k}(x)=\sum_{i=0}^k t(k,i)x^{2i}\;\; \text{and} \;\; U_{2k}(x)=\sum_{i=0}^k u(k,i)x^{2i}, \]
	where 
	\begin{align*}
		t(k,i)=(-1)^{k-i}2^{2i}\frac{k}{k+i}{k+i \choose k-i}, \;\;\;\; \; u(k,i)=(-1)^{k-i}2^{2i}{k+i \choose k-i}.
	\end{align*}	
	\noindent Note that $t(k,0)=(-1)^k=u(k,0)$, $t(k,k)=2^{2k-1}$ and $u(k,k)=2^{2k}$ for any $k$.  Here we emphasize that 
	 $$t(0,0)= \frac{1}{2}.$$

	\begin{theorem}\label{thm:L-to-LN}
		\begin{itemize}
			\item[(i)]For any even character $\chi\Mod N$ and any integer $n\geq 1$, we have
			\begin{align}\label{equ:L-even-values}
				L(2n,\chi)=\frac{\pi^{2n}}{ 2(2n-1)!N^{2n}} \sum_{i=1}^{n}a_{n,i} L_{\Z/N\Z}(i,\chi),
			\end{align}
			where  \begin{align}\label{eq:coefficient_a_ni}
				a_{n,i}=2\sum_{l=1}^{i}(-1)^{n-l}t(n-l,n-i)\bar{A}(2n,l), \;\text{for $i\in \{1,2,\cdots,n\}$}.
			\end{align}
			\item[(ii)] For any odd character $\chi \Mod N$ and any integer $n\geq 1$, we have
			\begin{equation}\label{equ:L-odd-values}
				L(2n+1,\chi)=\frac{\pi^{2n+1}}{ (2n)!N^{2n+1}} \sum_{i=1}^{n}b_{n,i}\widetilde{L}_{\Z/N\Z}(i,\chi),
			\end{equation}
			where  \begin{align}
				b_{n,i}=\sum_{l=1}^{i}(-1)^{n-l}u(n-l,n-i)\bar{A}(2n+1,l)=i a_{n, i}, \;\;\text{for $i\in \{1,2,\cdots,n\}$}.
			\end{align}
		\end{itemize}
		
	\end{theorem}
	\begin{proof}
		(i).	It is enough to calculate the following sum for an arbitrary variable $x$ 
		\begin{align*}
			&\sum_{l=1}^{2n-1}\cos\big( 2(n-l)x\big)\bar{A}(2n,l)\\
			=&\bar{A}(2n,n)+2 \sum_{l=1}^{n-1} \cos\big( 2(n-l)x\big)\bar{A}(2n,l)\\
			=&\bar{A}(2n,n)+2 \sum_{l=1}^{n-1} \bar{A}(2n,l)\sum_{i=0}^{n-l}(-1)^{n-l}t(n-l,i)\sin^{2i}x\\
			=&\bar{A}(2n,n)+2\sum_{l=1}^{n-1}\bar{A}(2n, l) +2 \sum_{l=1}^{n-1} \bar{A}(2n,l)\sum_{i=1}^{n-l}(-1)^{n-l}t(n-l,i)\sin^{2i}x\\			
			=&(2n-1)!+2\sum_{i=1}^{n-1}\sin^{2i}x\sum_{l=1}^{n-i}(-1)^{n-l}t(n-l,i)\bar{A}(2n,l),
		\end{align*}
		where the last equality follows from $t(n-l,0)=(-1)^{n-l}$ and 
		$$\sum_{l=1}^{2n-1}\bar{A}(2n,l)=(2n-1)!.$$

		For (ii), it is also sufficient to compute the following sum for an arbitrary variable $x$
		\begin{align*}
			&\sum_{l=1}^{2n}\cos\big( (2n-2l+1)x\big)\bar{A}(2n+1,l)\\
			=&2 \sum_{l=1}^{n} \cos\big( (2n-2l+1)x\big)\bar{A}(2n+1,l)\\
			=&2 \sum_{l=1}^{n} \bar{A}(2n+1,l)\sum_{i=0}^{n-l}(-1)^{n-l}u(n-l,i)\sin^{2i}x\cos x\\
			=&2\sum_{i=0}^{n-1}\sin^{2i}x\cos x\sum_{l=1}^{n-i}(-1)^{n-l}u(n-l,i)\bar{A}(2n+1,l).
		\end{align*}
		
		The relation between $a_{n,i}$ and $b_{n,i}$ follows from 
		\begin{align}\label{derivative-series}
			\sum_{k=-\infty}^{+\infty}\frac{1}{(2k\pi+\alpha)^{2n+1}}=\frac{-1}{2n}\frac{d}{d\alpha}\left( \sum_{k=-\infty}^{+\infty}\frac{1}{(2k\pi+\alpha)^{2n} } \right).
		\end{align}
		Indeed, by Theorem \ref{Mercer-torus}, the left hand side of \eqref{derivative-series} is 
		\begin{align*}
			\sum_{k=-\infty}^{+\infty}\frac{1}{(2k\pi+\alpha)^{2n+1}}&=\frac{1}{(2n)!2^{2n+1}\sin^{2n+1}(\alpha/2)}
			\sum_{l=1}^{2n} \cos\left((n-l+\frac{1}{2})\alpha\right) \overline{A}(2n+1,l)\\
			&=\frac{1}{(2n)!2^{2n}}\sum_{i=0}^{n-1}b_{n,n-i}\frac{\cos(\alpha/2)}{\sin^{2n-2i+1}(\alpha/2)}.
		\end{align*}
		The right hand side of \eqref{derivative-series} is 
		\begin{align}
			\frac{-1}{2n}\frac{d}{d\alpha}\left( \sum_{k=-\infty}^{+\infty}\frac{1}{(2k\pi+\alpha)^{2n} } \right)&=\frac{-1}{2n}\frac{d}{d\alpha}\left(    \frac{1}{2^{2n}(2n-1)!} \sum_{i=0}^{n-1}a_{n,n-i}\frac{1}{\sin^{2n-2i}(\alpha/2)} \right)\\
			&=\frac{1}{(2n)!2^{2n}}\sum_{i=0}^{n-1}(n-i)a_{n,n-i}\frac{\cos(\alpha/2)}{\sin^{2n-2i+1}(\alpha/2)}.
		\end{align}
		The claim follows by comparing the coefficients.
	\end{proof}
	
	\begin{remark}
		It is clear that the coefficients $a_{n,i}$ and $b_{n,i}$  are pure combinatorial and completely independent of the character $\chi \Mod N$.   We will show in  Proposition \ref{pro:A-and-C} that $a_{n, i}$ (hence also $b_{n, i}$) are
		in fact determined by the coefficients $c_{n, i}$ of the Laurent expansion of $\frac{1}{\sin^{2n} x}$  in Lemma \ref{Lemma_Zagier}.
	\end{remark}
	
	\noindent We also have a similar formula for $\zeta(2n)$ in terms of spectral zeta functions of the cycle graphs.
	\begin{theorem} \label{thm:zeta-spectral values}
		For any integers $N\geq 2$ and $n\geq 1$, we have
		\begin{align}
			(N^{2n}-1)\zeta(2n)=\frac{\pi^{2n}}{2(2n-1)!}\sum_{i=1}^{n}a_{n,i} \zeta_{\Z/N\Z}(i),
		\end{align}
		where $a_{n,i}, 1\leq i \leq n$, are given in \eqref{eq:coefficient_a_ni}.
	\end{theorem}
	\begin{proof} For each $d\mid N$, $d>1$,  we take the principal character $\chi_{0,d}$ on $\Z/d\Z$ in Theorem \ref{thm:L-to-LN}, we have 
		\[ d^{2n} L(2n, \chi_{0,d})=\frac{\pi^{2n}}{2(2n-1)!}\sum_{i=1}^{n}a_{n,i} L_{\Z/d\Z}(i,\chi_{0,d}).\]
		Noting that 
		\[ \zeta_{\Z/N\Z}(k)=\sum_{d\mid N} L_{\Z/d\Z}(k,\chi_{0,d}),\]
		we obtain 
		\begin{align}
			{\sum_{d\mid N}}' d^{2n} L(2n, \chi_{0,d})=\frac{\pi^{2n}}{2(2n-1)!}\sum_{i=1}^{n}a_{n,i} \zeta_{\Z/N\Z}(i),
		\end{align}
		where the sum $\sum'$ is taken over all divisors $d> 1$ of $N$.
		Therefore the assertion follows from the following equalities:
		\begin{align}
			{\sum_{d\mid N}}' d^{2n} L(2n, \chi_{0,d})=-\zeta(2n)+\sum_{d\mid N} d^{2n} \prod_{p\mid d}(1-\frac{1}{p^{2n}})\zeta(2n)=(N^{2n}-1)\zeta(2n).
		\end{align} 
	\end{proof}

	\noindent The above theorem gives us another way to compute the constants $\{a_{n,i}\}$ in terms of the Laurent expansion coefficients $c_{n, i}$ of $\frac{1}{\sin^{2n} x}$ defined in Lemma \ref{Lemma_Zagier}. Let 
	
	\begin{align*}
		A=\begin{pmatrix}
			a_{1,1} & 0 &\cdots&0 & 0\\
			a_{2,1} & a_{2,2} &\cdots &0& 0\\
			\vdots &\vdots & \ddots &\vdots &\vdots\\
			a_{n-1,1}&a_{n-1,2}&\cdots &a_{n-1,n-1}&0\\
			a_{n,1}&a_{n,2} &\cdots &a_{n,n-1} &a_{n,n}
		\end{pmatrix}
		&
		\mathrm{and\;}C=\begin{pmatrix}
			c_{1,1} & 0 &\cdots&0 & 0\\
			c_{2,1} & c_{2,2} &\cdots &0& 0\\
			\vdots &\vdots & \ddots &\vdots &\vdots\\
			c_{n-1,1}&c_{n-1,2}&\cdots &c_{n-1,n-1}&0\\
			c_{n,1}&c_{n,2} &\cdots &c_{n,n-1} &c_{n,n}
		\end{pmatrix}
	\end{align*}
	be the $ n \times n $ lower triangular matrices given by the constants $\{a_{n,i}\}$ and $\{c_{n,i}\}$ respectively.  We also let $B$ be the $ n \times n $ lower triangular matrix
	
	\begin{align*}
		B=\begin{pmatrix}
			b_{1,1} & 0 &\cdots&0 & 0\\
			b_{2,1} & b_{2,2} &\cdots &0& 0\\
			\vdots &\vdots & \ddots &\vdots &\vdots\\
			b_{n-1,1}&b_{n-1,2}&\cdots &b_{n-1,n-1}&0\\
			b_{n,1}&b_{n,2} &\cdots &b_{n,n-1} &b_{n,n}
		\end{pmatrix}.
	\end{align*}	
	Note that $B=A\cdot \mathrm{diag}\{1, 2, \cdots, n\}$ since $b_{n, i}=i a_{n, i}$.

	\begin{proposition}\label{pro:A-and-C}
		We have 
		\begin{equation}
			A\cdot C=\mathrm{diag}\{1!,3!,\cdots, (2n-1)!\}. 
		\end{equation}
	\end{proposition}
	\begin{proof}
		Using Euler's formula on even zeta values \eqref{equ:Eulerzeta} and Lemma \ref{Lemma_Zagier},  we have 
		\begin{align}
			&(N^{2n}-1)\frac{(-1)^{n+1}2^{2n}B_{2n}}{2n}=\sum_{i=0}^{n-1}a_{n,n-i}\sum_{s=0}^{n-i}\frac{(-1)^{s+1}2^{2s}B_{2s}}{(2s)!}c_{n-i,s}N^{2s}\\
			&=-\sum_{i=0}^{n-1}a_{n,n-i}c_{n-i,0}+\sum_{s=1}^{n} \frac{(-1)^{s+1}2^{2s}B_{2s}}{(2s)!} N^{2s}\sum_{i=0}^{n-s}a_{n,n-i}c_{n-i,s}. \label{equ:A-to-C}
		\end{align}
		Since both sides are polynomials in $N$,  by comparing the coefficients of $N^{2s}$, we obtain
		\[  \sum_{i=0}^{n-s}a_{n,n-i}c_{n-i,s}=\begin{cases*}
			(2n-1)! & if $s=n$,\\
			0 & if $1\leq s \leq n-1$.
		\end{cases*}\]
		The conclusion follows, since Theorem \ref{thm:zeta-spectral values} holds for any $n$.
	\end{proof}
	\vskip 3mm
	
	\noindent As a consequence, we may also write the special values of $L_{\Z/N\Z}(s,\chi)$ in terms of those of Dirichlet $L$-function.

	\begin{proposition}\label{spectral-to-zeta-value}
		(i)For an even character $\chi \Mod N$ and an integer $n\geq 1$, we have
		\begin{equation}\label{eq:graph-to-Dirichlet}
			L_{\Z/N\Z}(n,\chi)=2\sum_{i=1}^{n}c_{n,i}\frac{N^{2i}}{\pi^{2i}}L(2i,\chi)
		\end{equation}
		and in particular,
		\begin{equation}\label{eq:spectral-zeta-to-zeta}
			\zeta_{\Z/N\Z}(n)=2\sum_{i=1}^{n}c_{n,i}\frac{(N^{2i}-1)}{\pi^{2i}}\zeta(2i).
		\end{equation}
		(ii)  For an odd character $\chi\Mod N$ and an integer $n\geq 1$, we have
		\begin{equation}
			\widetilde{L}_{\Z/N\Z}(n,\chi)=\frac{2}{n}\sum_{i=1}^{n}i c_{n,i}\frac{N^{2i+1}}{\pi^{2i+1}}	L(2i+1,\chi).
		\end{equation}
	\end{proposition}
	\begin{proof}
		(i) Note that Theorem \ref{thm:L-to-LN} could also be written as a matrices product
		\begin{align}
			&\big{(}  \frac{2\cdot 1! N^2}{\pi^2}L(2,\chi),  \frac{2 \cdot3! N^4}{\pi^4}L(4,\chi), \cdots, \frac{2\cdot (2n-1)! N^{2n}}{\pi^{2n}}L(2n,\chi) \big{)}^T \notag\\
			=&A\cdot \big{(}{L_{\Z/NZ}(1,\chi)},{L_{\Z/NZ}(2,\chi)}, \cdots, {L_{\Z/NZ}(n,\chi)}\big{)}^T.
		\end{align}
		Then the claims follow from the computation of $A^{-1}$ in Proposition  \ref{pro:A-and-C} and Theorem \ref{thm:zeta-spectral values}.
		The proof of (ii) is similar to the above by noting that  
		$$B^{-1}=\mathrm{diag}\{1, \frac{1}{2}, \cdots, \frac{1}{n}\} \cdot A^{-1}=  \mathrm{diag}\{1, \frac{1}{2}, \cdots, \frac{1}{n}\} \cdot C \cdot \mathrm{diag}\{1, \frac{1}{3!}, \cdots, \frac{1}{(2n-1)!}\}.$$
	\end{proof}
	\begin{remark}
		Note that  the equality \eqref{eq:spectral-zeta-to-zeta} is just another formulation of Lemma \ref{Lemma_Zagier} except the constant term with respective the variable $N$.
	\end{remark}
	As an analogue of Zagier's formula for the Verlinde number in  Lemma \ref{Lemma_Zagier}, we have the following results:
	\begin{corollary}\label{D-to-G}
		\begin{itemize}
			\item[(i)]For an even primitive character $\chi\Mod N$ and an integer $n\geq 1$, we have
			\begin{align}
				L_{\Z/N\Z}(n,\chi)=\sum_{j=1}^{N}\frac{\bar{\chi}(j)}{G(\bar{\chi})} \sum_{i=1}^{n} \frac{(-1)^{i-1}2^{2i}{\mathbf{B}}_{2i}(\frac{j}{N})}{(2i)!}  c_{n,i}N^{2i},
			\end{align}
			where $\mathbf{B}_{k}(x)$ is the $k$-th Bernoulli polynomial and $G(\bar{\chi})$ is the Gauss sum associated to $\bar{\chi}$.
			\item[(ii)]For an odd primitive character $\chi\Mod N$ and an integer $n\geq 1$, we have
			\begin{align}
				\tilde{L}_{\Z/N\Z}(n,\chi)=\frac{1}{n}\sum_{j=1}^{N}\frac{\bar{\chi}(j)}{G(\bar{\chi})} \sum_{i=1}^{n} \frac{(2\sqrt{-1})^{2i+1}{\mathbf{B}}_{2i+1}(\frac{j}{N})}{(2i+1)!}  ic_{n,i}N^{2i+1}.
			\end{align}
		\end{itemize}
	\end{corollary}
	\begin{proof}
		Combining the identities  \eqref{Leopoldt} with \eqref{eq:graph-to-Dirichlet}, we obtain the assertions by noting that $G(\bar{\chi})G({\chi})=N$.
	\end{proof}
	
	\vskip 3mm

	\noindent   In the end of this section, we give some corollaries of  Theorem ~\ref{thm:L-to-LN}. As we mentioned before, our main Theorems \ref{thm:zeta-values} and 
	\ref{thm:L-to-LN} hold for any characters which are not necessarily primitive.  This makes it easier to handle some analytic number theory calculations. As an example, we give 
	the following simple results on the mean value of $L(n,\chi)$ as a direct consequence. We point out that one can get higher order power mean values of $L(n,\chi)$ similarly.
	
	\begin{corollary}\label{cor:mean}
		For an integer $n\geq 1$, we have
		\begin{equation}
			\lim\limits_{N\to \infty}\frac{2}{\varphi(N)}\sum_{\chi(-1)=(-1)^n} L(n,\chi)=1,
		\end{equation}
		where  the sum runs over all  Dirichlet characters of $\Z/N\Z$ with $\chi(-1)=(-1)^n$ and $\varphi(N)$ is the Euler $\phi$-function of $N$.
	\end{corollary}
	
	\begin{proof}
		First note that the following identities
		\begin{align*}
			\sum_{\chi\; \mathrm{even}} \chi(m)=\begin{cases}
				\varphi(N)/2 & \mathrm{if } \; m \equiv \pm 1 \mod N\\
				0 & \mathrm{otherwise,}
			\end{cases}
		\end{align*}
		and
		\begin{align*}
			\sum_{\chi\; \mathrm{odd}} \chi(m)=\begin{cases}
				\varphi(N)/2 & \mathrm{if } \; m \equiv  1 \mod N\\
				-\varphi(N)/2 & \mathrm{if } \; m \equiv  -1 \mod N\\
				0 & \mathrm{otherwise}.
			\end{cases}
		\end{align*}
		Combining with the formulae \eqref{equ:L-even-values}  and \eqref{equ:L-odd-values} in the above Theorem, 	we have 
		\begin{equation}
			\frac{2}{\varphi(N)}\sum_{\chi\; \mathrm{even}} L(2n,\chi)=\frac{\pi^{2n}}{(2n-1)! N^{2n}}\sum_{i=0}^{n-1}\frac{a_{n,n-i}}{\sin^{2n-2i}(\pi/N)}
		\end{equation}
		and 
		\begin{equation}
			\frac{2}{\varphi(N)}\sum_{\chi\; \mathrm{odd}} L(2n+1,\chi)=\frac{2\pi^{2n+1}}{(2n)! N^{2n+1}}\sum_{i=0}^{n-1}\frac{b_{n,n-i}\cos(\pi/N)}{\sin^{2n-2i+1}(\pi/N)}.
		\end{equation}
		The assertions follow from the facts that $a_{n,n}=(2n-1)!, b_{n,n}=n\cdot(2n-1)!$ and $\sin (\frac{\pi}{N}) \sim \frac{\pi}{N}$ as $N \to \infty$.
		
		For the case $\chi(-1)=-1$ and $n=1$, by using formula \eqref{chi1},  we similarly get 
		\begin{equation}
			\frac{2}{\varphi(N)}\sum_{\chi\; \mathrm{odd}} L(1,\chi)=\frac{\pi}{N}\cot(\frac{\pi}{N}),
		\end{equation}
		which clearly goes to 1 as $N\to \infty$.
	\end{proof}
	
	\section{Concluding remarks}
	Here are a few possible interesting topics related to the materials discussed in this paper for further investigations:\\
	
	\noindent \textit{Evaluation of the cycle graph $L$-function.\footnote{After the first version of this paper was posted on arXiv at the end of the year 2022, this problem has been solved for primitive characters very recently 
	by Jorgenson, Karlsson and Smajlovi\'c  in \cite{JKS}.}} 
	In Corollary \ref{D-to-G}, we show, for  a primitive character,  how to get the special values of $L_{\Z/N\Z}(n, \chi)$	 from those of 
	$L(n, \chi)$.  Note that $L_{\Z/N\Z}(n, \chi)$ is a finite sum of elementary trigonometric functions twisted with a character.  It is natural to ask for a direct way to calculate $L_{\Z/N\Z}(n, \chi)$ instead of using the known formulae of Dirichlet $L$-functions.	
	If so, then we aslo find another way to evaluate special values of Dirichlet $L$-functions.
	\vskip 5mm
	\noindent \textit{Relations with Iwasawa theory.} It is well known that the special values of Dirichlet $L$-functions determine the $p$-adic $L$-function by interpolation.  Then,  given the fact that these values are completely determined by
	the values of $L_{\Z/N\Z}(s, \chi)$ as showed in Theorem \ref{thm:L-to-LN},  what is the  $p$-adic analogue of $L_{\Z/N\Z}(n, \chi)$ and what is its relation with the $p$-adic $L$-function in Iwasawa theory?	
	\vskip 5mm
	\noindent \textit{Linear independence of half integer values (over  $\Q$).} For an odd integer $n$,  we recall that   $\zeta_{\Z/N\Z}(\frac{n}{2})=  \sum_{m=1}^{N-1} \frac{1}{\sin^{n}(\frac{m\pi}{N})}. $	
	Does the following set
	$$\Big{\{}\zeta_{\Z/N\Z}(\frac{1}{2}),  \zeta_{\Z/N\Z}(\frac{3}{2}), \cdots,  \zeta_{\Z/N\Z}(\frac{2N+1}{2}) \Big{\}}$$	
	tend to span as large dimension as possible over $\Q$?  This might be viewed as an analogue question of the wildly believed conjecture that 
	$\zeta(3), \zeta(5), \cdots, \zeta(2n+1), \cdots ,$ are linearly independent over $\Q$.
	
	\vskip 5mm
	\noindent \textit{Higher dimensional graph torus.}  Theorem \ref{thm:zeta-spectral values} establishes a relation between special values of univariate zeta function with those of spectral zeta of the one dimensional torus $C_N$. 
	Do we have analogue results of Theorem \ref{thm:zeta-spectral values}  for the higher dimension torus $(\Z/N\Z)^d$? Is there any relation between special values of the spectral zeta function of a graph torus and of an Epstein zeta function?\\
	Of course, one can ask similar questions for more general manifolds (and with corresponding graphs) $\cdots$.
	\vskip 7mm
	
	\noindent 
	{\bf Acknowledgements.} 
	We would like to thank Anders Karlsson, Dylan M{\"u}ller, Daqing Wan, Jun Wang and Yichao Zhang for helpful discussions.  
	
	A part of this work was done during the last two named authors' stay at the Grand Rezen Golden Bay Hotel, Weihai. They enjoyed their stay in the hotel very much and thank all the staff, especially,  Mrs. Zou, for their hospitality and 
	for providing excellent working conditions.

\end{document}